\documentclass{amsart}

\usepackage{amsmath}
\usepackage{amscd}
\usepackage{amssymb}
 \usepackage{color}

\input xy
\xyoption{all}

\newcommand{\pd}{\partial}
\newcommand{\bC}{{\mathbb C}}

\newcommand{\bN}{{\mathbb N}}
\newcommand{\bP}{{\mathbb P}}

\newcommand{\bZ}{{\mathbb Z}}

\newcommand{\cE}{{\mathcal E}}
\newcommand{\cF}{{\mathcal F}}
\newcommand{\cM}{{\mathcal M}}

\newcommand{\cP}{{\mathcal P}}
\newcommand{\half}{\frac{1}{2}}

\newcommand{\cW}{{\mathcal W}}
\newcommand{\Mbar}{\overline{\cM}}
\newcommand{\vac}{|0\rangle}
\newcommand{\lvac}{\langle 0|}
\newcommand{\cor}[1]{\langle {#1} \rangle}
 \newcommand{\bp}{{\mathbf p}}
 
\DeclareMathOperator{\Id}{Id}

\DeclareMathOperator{\res}{res}

\newtheorem{theorem}{Theorem}[section]
\newtheorem{theorem/definition}{Theorem/Definition}[section]

\newtheorem{proposition}{Proposition}[section]
\newtheorem{lemma}{Lemma}[section]

\newtheorem{corollary}{Corollary}[section]

\newtheorem{conjecture}{Conjecture}[section]

\theoremstyle{remark}

\theoremstyle{definition}

\newcommand{\bml}{\begin{multline}}
\newcommand{\eml}{\end{multline}}

\newcommand{\bag}{\begin{align}}
\newcommand{\egn}{ \end{align}}

\newcommand{\be}{\begin{equation}}
\newcommand{\ee}{\end{equation}}
\newcommand{\bea}{\begin{eqnarray}}
\newcommand{\eea}{\end{eqnarray}}
\newcommand{\ben}{\begin{eqnarray*}}
\newcommand{\een}{\end{eqnarray*}}
\newcommand{\bet}{\begin{equation}
\begin{split}}
\newcommand{\eet}{\end{split}
\end{equation}}

\definecolor{yellow}{rgb}{1,1,0}
\definecolor{orange}{rgb}{1,.7,0}
\definecolor{red}{rgb}{1,0,0} \definecolor{blue}{rgb}{0,0,1}
\definecolor{white}{rgb}{1,1,1}

\definecolor{A}{rgb}{.75,1,.75}

\begin{document}

\title[Open String Invariants and Mirror Curve]
{Open String Invariants and Mirror Curve of the Resolved Conifold}
\author{Jian Zhou}
\address{Department of Mathematical Sciences\\Tsinghua University\\Beijing, 100084, China}
\email{jzhou@math.tsinghua.edu.cn}

\begin{abstract}
For the resolved conifold with one outer D-brane in arbitrary framing,
we present some results for the open string partition functions obtained by  some operator manipulations.
We prove some conjectures by Aganagic-Vafa and Aganagic-Klemm-Vafa
that relates such invariants to the mirror curve of the resolved conifold.
This establishes local mirror symmetry for the resolved confolds for holomorphic disc invariants.
We also verify an integrality conjecture of such invariants by Ooguri-Vafa in this case
and present closed formulas for some Ooguri-Vafa type invariants in genus $0$ and arbitrary genera.
\end{abstract}


\maketitle

\section{Introduction}

The resolved conifold is probably the most studied noncompact Calabi-Yau $3$-fold,
comparable as its compact counterpart the quintic in $\bP^4$.
It provides testing ground for new developments and is a rich source of many new discoveries.
In this work we will study it as the first  nontrivial local Calabi-Yau geometries
for which we try to establish local mirror symmetry in the new formalism of the local B-theory \cite{Mar, BKMP},
which involves open string invariants.

The local $A$-theory of the resolved conifold has played a crucial role in the development
of duality between topological string theory and Chern-Simons theory.
Witten \cite{Wit} proposed that the large N expansion of Chern-Simons link invariants
can be identified with the genus expansion of open string invariants of the deformed conifold $T^*S^3$,
which counts holomorphic maps from Riemann surfaces boundaries to $T^*S^3$,
with $h$ boundary components mapped to the Lagrangian submanifold $S^3$.
Gopakumar-Vafa \cite{Gop-Vaf} proposed that after summing over the number $h$ of boundary components,
one gets from the large $N$ expansion of the $3$-manifold invariant
of $S^3$ the genus expansion of closed string invariants of the resolved conifold.
Ooguri and Vafa \cite{Oog-Vaf} extended this and proposed that the large $N$ expansion of link invariants
corresponds to open string invariants of the resolved conifold,
associated with some Lagrangian submanifolds corresponding to the link.
Mari\~no-Vafa \cite{Mar-Vaf} extended this further by considering framed knots.
The duality has been extended to other local Calabi-Yau geometries \cite{AMV} and a formalism called
the topological vertex \cite{AKMV} has been developed to compute open and closed string invariants
of local toric Calabi-Yau geometries.

Much of these developments in the physics literature has been made mathematically rigorous,
using the method of localization.
This method reduces the calculations to Hodge integrals which are intersection numbers
on the Deligne-Mumford moduli spaces.
A byproduct in \cite{Mar-Vaf} is a conjectural closed formula for some Hodge integrals, by comparing with the formal localization
calculations performed for the resolved conifold by Katz and Liu \cite{Kat-Liu}.
This has played no role in later physical development,
but it turns out to provide the crucial clue for the mathematical development.
The formula for Hodge integrals conjectured by Mari\~no and Vafa has been proved \cite{LLZ1, Oko-Pan}
and generalized \cite{Zho1, LLZ2, LLLZ},
and all such formulas have played an important role in the mathematical development of
the calculations of open and closed string invariants of local Calabi-Yau geometries \cite{Zho2, LLLZ}.
Note mathematically the open string invariants is very difficult to define in symplectic geometry.
Li and Song \cite{Li-Son} notice that in the case of the resolved conifold
one can use the relative moduli space in algebraic geometry to define open string invariant.
This is also the approach taken in the mathematical theory of the topological vertex \cite{LLLZ}
to bypass the difficulties in the definition of open string invariants for general local Calabi-Yau geometries.
So now we have a mathematically rigorous way to define and compute certain open string invariants
of local Calabi-Yau geometries in algebraic geometry,
which should have a parallel theory in symplectic geometry as shown in the resolved conifold case in \cite{Li-Son}.

Local B-theory and local mirror symmetry  in genus $0$ have also been studied extensively in physics literature.
For closed strings, see e.g. \cite{CKYZ}.
For open string, Aganagic and Vafa \cite{AV} conjectured a relationship between
genus $0$ open string invariants and the equation for the mirror curve
given by the Hori-Iqbal-Vafa construction  \cite{HIV}.
This was extended by Aganagic, Klemm and Vafa \cite{AKV} to include the effect of framing,
and the case of the resolved conifold was treated in \cite{Mar-Vaf}.
Klemm and Zaslow \cite{Kle-Zas} studied local B-theory from the point of view of holomorphic anomaly equation
\cite{BCOV}.
For another approach see \cite{ADKMV}.
Recently there has appeared a new formalism of the local B-theory in arbitrary genera.
Extending results in the theory of matrix models,
Eynard-Orantin \cite{EO} developed a formalism that recursively defines some differentials
from a plane algebraic curve.
See also their more recent work \cite{EO2}.
These differentials correspond to $n$-point open string partition functions
in arbitrary genus.
Mari\~no \cite{Mar} proposed that this formalism can be used to define unambiguously the local B-theory
for local toric Calabi-Yau geometries,
where the algebraic curves are taken to be the mirror curves.
Bouchard, Klemm, Mari\~no, Pasquetti \cite{BKMP} extended this proposal to more general setting based on \cite{AKV},
including some conjectures on the framing transformation and the effects of phase changes of the D-branes.
Bouchard and Mari\~no \cite{Bou-Mar} made the proposal more explicit in the case of $\bC^3$ with one D-brane,
incorporating the ideas from \cite{AKV}.
Mathematically,
this involves one-partition Hodge integrals in the Mari\~no-Vafa formula \cite{Mar-Vaf, LLZ1, Oko-Pan}.
They also derived a related conjecture for Hurwitz number in the same paper.

Borot-Eynard-Safnuk-Mulase \cite{BEMS}
and Eynard-Safnuk-Mulase \cite{EMS} have proved Bouchard-Mari\~no Conjecture for Hurwitz numbers by two different methods.
Based on the method in \cite{EMS},
two proofs, one by Chen \cite{Che} and one by the author \cite{Zho5}, have appeared for Bouchard-Mari\~no Conjecture
for the case of $\bC^3$ with one D-brane.
The proof for the case of $\bC^3$ with three D-branes,
i.e.,
the case of topological vertex \cite{AKMV, LLLZ},
has been given by the author in a more recent work \cite{Zho6}.
Our strategy for the study the case of more general local Calabi-Yau geometries
is to generalize the earlier results.
It has three steps:
\begin{enumerate}
\item[Step 1.] Compute the genus zero one point functions of open string invariants and relate it to the mirror curve.
\item[Step 2.] Compute the genus zero two point functions of open string invariants and relate it to the Bergman kernel
of the mirror curve.
\item[Step 3.] Establish the Eynard-Orantin type recursion relations,
favorably by the cut-and-join equations.
\end{enumerate}

In this work we focus on the first step for the resolved conifold,
leaving the other two steps to subsequent work in progress.
We will start with the calculation of the open string invariants by topological vertex.
Such calculations involve summations over partitions and so are often very complicated.
A key technique we use in this work is to transform
thecalculations to vacuum expectation values of some operators
developed by Okounkov-Pandharipande \cite{Oko-Pan1, Oko-Pan2}
originally for Hurwitz numbers.
This was first proposed in an earlier work \cite{Zho4} for the case of inner brane,
whereas in this work we deal with the case of out brane.
However we go one step further by finding simple closed formula for one-point functions of open string invariant
and hence establish the conjectured relation with the mirror curve proposed in \cite{AV, AKV}.
This establishes local mirror symmetry for the resolved conifold
in the case of holomorphic disc invariants.
The case of inner brane can be treated similarly with details to be presented in a separate paper.
As another application,
we also check the Ooguri-Vafa Integrality Conjecture \cite{Oog-Vaf} for one-point functions in arbitrary framing.
We present a closed formula for holomorphic disc numbers and a closed formula for some integral invariants
in arbitrary genera.

Another application of our results is that it leads to a proof of the full Mari\~no-Vafa Conjecture
that identifies open string invariants of the resolved conifold with the link invariants of the unknot,
i.e., the quantum dimensions.
The details will appear in a forthcoming paper.

The next step in establishing the local mirror symmetry in the BKMP Conjecture
is to compute the two-point functions in genus zero
and relate it to the Bergman kernel of the mirror curve.
Our method in this paper can also be used to evaluate the two-point functions,
but the combinatorics is much more complicated.
We will address this in a subsequent work.

Note in this paper an important simplifying trick we have exploited
is that after the conversion from the topological vertices to Hodge integrals,
one can choose some parameters suitably so that most of the terms vanish and for the nonvanishing terms
one can use the Mumford's relations.
For more general local Calabi-Yau geometries this may not work anymore.

The rest of the paper is arranged as follows.
In Section 2 we start with the topological vertex computations of open string invariants of
the resolved conifold in general framing,
and show that after conversion to 2-partition Hodge integrals,
some results in \cite{Zho1} can be used to simplify the calculations.
In Section 3 we adapt the proposal in the inner brane case in \cite{Zho4}
to the outer brane case,
and reformulate the partition functions in terms of operators introduced in \cite{Oko-Pan1}.
We indicate the method of evaluating the $n$-point functions by operator manipulations,
and carry out in Section 4 the calculation of one-point functions by this method.
Closed formulas in genus $0$ and in general genera are presented in this section.
As an application,
we discuss in Section 5 the Ooguri-Vafa integrality properties of disc invariants of the resolved conifold
in general framing.
In the final Section 6,
we establish the Aganagic-Vafa Conjecture that relates counting of open string disc invariants to the mirror curve,
and the Aganagic-Klemm-Vafa Conjecture on framing transformation, both in the resolved conifold case.

\vspace{.1in}
{\em Acknowledgements}.
This research is partially supported by two NSFC grants (10425101 and 10631050)
and a 973 project grant NKBRPC (2006cB805905).

\section{Open String Amplitude of the Resolved Conifold with One outer Brane}

In this section we will start with the calculation by topological vertex \cite{AKMV, LLLZ}
of the open string amplitude of the resolved conifold.
We reformulate it by a formula that relates the relevant topological vertex to two-partition Hodge integrals
\cite{Zho1, LLZ2}.
We then show that by choosing the parameters in the Hodge integrals suitably
the computations can be greatly simplified.

\subsection{Open string amplitudes by the topological vertex}
By the theory of the topological vertex \cite{AKMV, LLLZ},
the open string amplitude for the resolved conifold with one outer brane and framing $a$
can be computed as follows:
\be
Z^{(a)}(\lambda; t; \bp)
= \sum_{\mu, \nu, \eta} q^{a \kappa_\mu/2}  C_{(0), \mu, \nu^t}(q) \cdot Q^{|\nu|} \cdot
C_{\nu,(0), (0)}(q) \cdot
\frac{\chi_{\mu}(\eta)}{z_{\eta}\sqrt{-1}^{l(\eta)}} p_{\eta}(x),
\ee
where $q = e^{\sqrt{-1}\lambda}$ and $Q = - e^{-t}$.
The topological vertices here can be rewritten as follows (cf \cite{Zho3}):
\begin{align}
C_{(0), \mu, \nu^t}(q) & = q^{-\kappa_\nu/2} W_{\mu,\nu}(q), &
C_{\nu, (0), (0)}(q) = W_{\nu}(q) = q^{\kappa_\nu/2} W_{\nu^t}(q),
\end{align}
so we have
\be \label{eqn:Za}
Z^{(a)}(\lambda;t;\bp) = \sum_{\mu, \nu, \eta} q^{a \kappa_\mu/2} W_{\mu\nu}(q) Q^{|\nu|}  W_{\nu^t}(q)\cdot
\frac{\chi_{\mu}(\eta)}{z_{\eta}\sqrt{-1}^{l(\eta)}} p_{\eta}(x).
\ee
We will not recall the complicated definitions of $W_{\mu,\nu}(q)$ here,
since we will not use it below to carry out the computations.
Instead,
we will rewrite it in terms of Hodge integrals and use a trick
that greatly simplifies the summations over partitions.

\subsection{Open string amplitudes by two-partition Hodge integrals}

For a pair of partitions $(\mu^+, \mu^-) \in \cP^2_+$ (one of which might be empty),
consider the following  Hodge integrals:
\begin{eqnarray*}
&& G_{\mu^+, \mu^-}(x, y, -x-y) \\
& = & - \frac{\sqrt{-1}^{l(\mu^+)+l(\mu^-)}}{z_{\mu^+} \cdot z_{\mu^-}} \cdot
\sum_{g \geq 0} \lambda^{2g-2+l(\mu^+)+l(\mu^-)} \\
&& \cdot \int_{\Mbar_{g, l(\mu^+)+l(\mu^-)}}
\frac{\Lambda_{g}^{\vee}(x)\Lambda^{\vee}_{g}(y)\Lambda_{g}^{\vee}(-x - y)}
{\prod_{i=1}^{l(\mu^+)} \frac{x}{\mu_i^+} \left(\frac{x}{\mu^+_i} - \psi_i\right)
\prod_{j=1}^{l(\mu^-)} \frac{y}{\mu_i^-}\left(\frac{y}{\mu^-_j} - \psi_{j+l(\mu^+)}\right)} \\
&& \cdot \left[xy(x+y)\right]^{l(\mu^+)+l(\mu^-)-1}
\cdot \prod_{i=1}^{l(\mu^+)} \frac{\prod_{a=1}^{\mu^+_i-1}
\left( \mu^+_iy + a x\right)}{\mu_i^+! x^{\mu_i^+-1}}
\cdot \prod_{i=1}^{l(\mu^-)} \frac{\prod_{a=1}^{\mu^-_i-1}
\left( \mu_i^- x + a y\right)}{\mu_i^-! y^{\mu_i^--1}}.
\end{eqnarray*}
The following formula was conjectured by the author \cite{Zho1} and proved in joint work with
C.-C. Liu and K. Liu \cite{LLZ2}:
\begin{eqnarray} \label{eqn:Zhou}
&& G^{\bullet}(\lambda; x, y, -x-y; p^+,p^-) = R^{\bullet}(\lambda; x, y, -x-y; p^+, p^-),
\end{eqnarray}
where
\begin{eqnarray*}
G^{\bullet}(\lambda; x, y,-x-y; p^+, p^-) & = &
\exp \left(
\sum_{(\mu^+, \mu^-) \in \cP_+^2} G_{\mu^+, \mu^-}(\lambda;x,y, -x-y)p^+_{\mu^+}p^-_{\mu^-}\right), \\
R^{\bullet}(\lambda;x, y, -x-y; p^+, p^-) & = & \sum_{\mu^{\pm},\nu^{\pm}}
\frac{\chi_{\nu^+}(\mu^+)}{z_{\mu^+}} \frac{\chi_{\nu^-}(\mu^-)}{z_{\mu^-}}
q^{(\frac{\kappa_{\nu^+} y}{x}  + \frac{\kappa_{\nu^-} x}{y})\lambda/2}
\cW_{\nu^+, \nu^-}(q) p^+_{\mu^+}p^-_{\mu^-}.
\end{eqnarray*}
Write
\be
G^{\bullet}(\lambda; x, y,-x,y; p^+, p^-) =
\sum_{(\mu^+, \mu^-) \in \cP_+^2} G^\bullet_{\mu^+, \mu^-}(\lambda;x,y, -x-y)p^+_{\mu^+}p^-_{\mu^-},
\ee
then by (\ref{eqn:Zhou}) we get:
\be
W_{\nu^1, \nu^2}(q)
= \sum_{|\mu^i|=|\nu^i|} \prod_{i=1}^2 \chi_{\nu^i}(\mu^i)
\cdot q^{-(\kappa_{\nu^1} w_2/w_1 + \kappa_{\nu^2} w_1/w_2)/2}
G_{\mu^1, \mu^2}^\bullet(\lambda; w_1, w_2, w_3),
\ee
where $w_3 = - w_1 - w_2$.
One can use this formula to rewrite (\ref{eqn:Za}) as follows:
\begin{multline} \label{eqn:ZaInG}
 Z^{(a)}(\lambda;t; \bp)
= \sum_{\mu, \xi, \eta, \epsilon}  q^{(a-w_1/w_2)\kappa_\mu/2}
\chi_{\mu}(\epsilon)
G^\bullet_{\xi\epsilon}(\lambda;w_1,w_2,w_3) \cdot z_{\xi} \\
\cdot Q^{|\xi|} (-1)^{|\xi|-l(\xi)} \cdot G_{\xi}^\bullet(\lambda;-w_1,-w_2, -w_3)
\cdot
\frac{\chi_{\mu}(\eta)}{z_{\eta}\sqrt{-1}^{l(\eta)}} p_{\eta}(x).
\end{multline}
Indeed,
\ben
&& Z^{(a)}(\lambda;t; \bp) \\
& = & \sum_{\mu, \nu, \xi, \eta, \epsilon,\phi}  q^{a\kappa_\mu/2} \chi_{\nu}(\xi) \chi_{\mu}(\epsilon) q^{-(\kappa_\nu w_2/w_1 + \kappa_\mu w_1/w_2)/2}
G^\bullet_{\xi\epsilon}(\lambda;w_1,w_2,w_3) \\
&& \cdot Q^{|\nu|} \chi_{\nu^t}(\phi) q^{-(\kappa_{\nu^t} w_2/w_1) /2} G_{\phi}^\bullet(\lambda;-w_1,-w_2, -w_3)
\cdot
\frac{\chi_{\mu}(\eta)}{z_{\eta}\sqrt{-1}^{l(\eta)}} p_{\eta}(x) \\
& = & \sum_{\mu, \xi, \eta, \epsilon,\phi}  q^{(a-w_1/w_2)\kappa_\mu/2}
\chi_{\mu}(\epsilon)
G^\bullet_{\xi\epsilon}(\lambda;w_1,w_2,w_3) \\
&& \cdot Q^{|\xi|} \sum_{\nu} \chi_{\nu^t}(\phi) \chi_{\nu}(\xi) \cdot G_{\phi}^\bullet(\lambda;-w_1,-w_2, -w_3)
\cdot
\frac{\chi_{\mu}(\eta)}{z_{\eta}\sqrt{-1}^{l(\eta)}} p_{\eta}(x) \\
& = &  \sum_{\mu, \xi, \eta, \epsilon}  q^{(a-w_1/w_2)\kappa_\mu/2}
\chi_{\mu}(\epsilon)
G^\bullet_{\xi\epsilon}(\lambda;w_1,w_2,w_3) \cdot z_\xi \\
&& \cdot Q^{|\xi|} (-1)^{|\xi|-l(\xi)} \cdot G_{\xi}^\bullet(\lambda;-w_1,-w_2, -w_3)
\cdot
\frac{\chi_{\mu}(\eta)}{z_{\eta}\sqrt{-1}^{l(\eta)}} p_{\eta}(x).
\een
In the above we have used the following facts:
\begin{align}
\kappa_{\nu^t} & = - \kappa_{\nu}, &
\chi_{\nu^t}(\phi) & = (-1)^{|\phi|-l(\phi)} \chi_\nu(\phi),
\end{align}
and the following orthogonality relations for characters:
\be
\sum_{\nu} \chi_\nu(\xi)\chi_\nu(\phi) = \delta_{\xi, \phi} z_\xi.
\ee

\subsection{Hodge integrals at special parameters}

In general it is very hard to evaluate $G^\bullet_{\xi\epsilon}(\lambda;w_1,w_2,-w_1-w_2)$.
Fortunately in \cite{Zho1} we notice that $G^\bullet_{\xi\epsilon}(\lambda;w_1,-w_1,0)$
can be easily evaluated.
In fact,
when $y=-x$,
$G_{\mu^+, \mu^-}(\lambda, x, y, -x-y)$ vanishes except for the following three cases.
Case 1.
\ben
&& G_{(n),(0)}(\lambda; x,-x, 0) \\
& = & -\frac{\sqrt{-1}}{n} \sum_{g \geq 0} \lambda^{2g-1} \int_{\Mbar_{g,1}}
\frac{\Lambda_{g}^{\vee}(x)\Lambda^{\vee}_{g}(-x)\Lambda_{g}^{\vee}(0)}
{\frac{x}{n} \left(\frac{x}{n} - \psi_1\right)}
\cdot  \frac{\prod_{a=1}^{n-1} \left( -nx + a x\right)}{n! x^{n-1}} \\
& = & \frac{(-1)^{n-1}}{\sqrt{-1}\lambda n^2} \sum_{g \geq 0} (n\lambda)^{2g}
\int_{\Mbar_{g,1}} \psi_1^{2g-2}\lambda \\
& = & \frac{(-1)^{n-1}}{n} \frac{1}{2\sqrt{-1} \sin (n\lambda/2)} = \frac{(-1)^{n-1}}{n \cdot [n]},
\een
where $[n] = q^{n/2} - q^{-n/2}$.
Here in the second equality we have used Mumford's relations:
\be
\Lambda_g^\vee(x) \Lambda_g^\vee(-x) = (-x)^g.
\ee
Similarly,
\ben
&& G_{(0),(n)}(\lambda; x,-x,0) = \frac{(-1)^{n-1}}{n \cdot [n]}.
\een
Another case where the $G_{\mu, \nu}(\lambda; x, -x, 0)$ may be nonvanishing is
\ben
&& G_{(n),(n)}(\lambda; x, -x) = \frac{1}{n}.
\een
Therefore,
\begin{multline} \label{eqn:G-1}
G^\bullet(\lambda;x, -x, 0; p^+, p^-) \\
= \exp \sum_{n=1}^\infty \biggl( \frac{(-1)^{n-1}}{n \cdot [n]} p_n(x^+)
+ \frac{(-1)^{(n-1)}}{n \cdot [n]} p_n(x^-) + \frac{1}{n} p_n(x^+)p_n(x^-) \biggr).
\end{multline}
One can then easily find
\bea
&& G^\bullet_{(1),(0)}(\lambda;x,-x, 0) =  G^\bullet_{(0),(1)}(\lambda; x,-x, 0) = \frac{1}{[1]}, \\
&& G^\bullet_{(2),(0)}(\lambda;x,-x, 0) =  G^\bullet_{(0),(2)}(\lambda;x,-x, 0) =
-  \frac{1}{2 \cdot [2]}, \\
&& G^\bullet_{(1^2),(0)}(\lambda;x,-x, 0) =  G^\bullet_{(0),(1^2)}(\lambda;x,-x, 0) =  \half \frac{1}{[1]^2}, \\
&& G^\bullet_{(1),(1)}(\lambda;x,-x, 0) = (\frac{1}{[1]^2} + 1),
\eea
etc.
Now we take $w_2 = - w_1$ and so $w_3 = 0$ in (\ref{eqn:ZaInG}):
\begin{multline} \label{eqn:ZaInG2}
 Z^{(a)}(\lambda;t; \bp)
= \sum_{\mu, \xi, \eta, \epsilon}  q^{(a+1)\kappa_\mu/2}
\chi_{\mu}(\epsilon)
G^\bullet_{\xi\epsilon}(\lambda;w_1,-w_1,0) \cdot z_\xi \\
\cdot Q^{|\xi|} (-1)^{|\xi|-l(\xi)} \cdot G_{\xi}^\bullet(\lambda;-w_1,w_1, 0)
\cdot
\frac{\chi_{\mu}(\eta)}{z_{\eta}\sqrt{-1}^{l(\eta)}} p_{\eta}(x).
\end{multline}
Hence one can apply (\ref{eqn:G-1}) to compute $Z^{(a)}(\lambda; t; \bp)$.

\section{Open String Amplitude of the Resolved Conifold by Operator manipulations}

In last section we have seen that the computation of the open string amplitude of the resolved conifold
involves complicated summations over partitions.
In this section we will follow the treatment of the inner brane case in \cite{Zho4} to reformulate it
using the operators in \cite{Oko-Pan1}.

\subsection{Reformulation in terms of symmetric functions}

We now rewrite the open string amplitude as the vacuum expectation value on the space $\Lambda$ of symmetric functions \cite{Mac}.
Recall the Newton functions $\{p_\mu\}$ and the Schur functions $\{s_\nu\}$ form additive bases of $\Lambda$,
and they are related by the character values of the symmetric groups:
\begin{align}
p_{\mu} & = \sum_{\nu} \chi_{\nu}(\mu) s_{\nu}, & s_{\nu} & =  \sum_\mu \frac{\chi_\nu(\mu)}{z_\mu} p_\mu.
\end{align}
Under the natural scalar product on $\Lambda$ one has:
\begin{align}
\langle p_\mu, p_\nu\rangle & = \delta_{\mu, \nu} z_{\mu}, &
\langle s_\mu, s_\nu\rangle & = \delta_{\mu, \nu}.
\end{align}
In particular,
\be
\langle p_\mu, s_\nu \rangle = \chi_\nu(\mu).
\ee
On $\Lambda$ one can introduce the following operators for nonzero integers $n$:
\be
\beta_n: \Lambda \to \Lambda, \qquad \beta_n (f) = \begin{cases}
p_{-n} \cdot f, & n < 0, \\
n \frac{\pd}{\pd p_n} f, & n > 0.
\end{cases}
\ee
These operators satisfy the following commutation relations:
\be
[\beta_m, \beta_n ] = m \delta_{m, -n} \Id_{\Lambda}.
\ee

Another useful operator on $\Lambda$ is the cut-and-join operator
\begin{align}
K: & = \half \sum_{i,j=1}^\infty \big(p_{i+j} ij\frac{\pd^2}{\pd p_i\pd p_j}
+ p_ip_j (i+j) \frac{\pd}{\pd p_{i+j}} \big) \\
& = \half \sum_{i,j =1}^\infty (\beta_{-(i+j)} \beta_i\beta_j + \beta_{-i}\beta_{-j} \beta_{i+j}). \nonumber
\end{align}
The Schur functions are eigenvectors of this operator:
\be
K s_\mu = \half \kappa_\mu s_\mu.
\ee

For an linear operator $A: \Lambda \to \Lambda$,
It vacuum expectation value is
\be
\cor{A} = \lvac A \vac,
\ee
where $\vac = 1 \in \Lambda$.
One clearly has
\be
\beta_n \vac = 0, \qquad n > 0.
\ee
For a partition $\mu = (\mu_1, \dots, \mu_l)$,
we write
\begin{align}
\beta_\mu & = \prod_{i=1}^l \beta_{\mu_i}, & \beta_{-\mu} & = \prod_{i=1}^l \beta_{-\mu_i}.
\end{align}
It is easy to see that
\be
\cor{ \beta_\mu \beta_{-\nu}} = \delta_{\mu, \nu} z_{\mu}.
\ee
It follows from this identity that
\be
\cor{\exp (\sum_{n=1} \frac{a_n}{n} \beta_n) \cdot \exp ( \sum_{n=1}^\infty
\frac{b_n}{n} \beta_{-n} ) }
= \exp \sum_{n=1}^\infty \frac{a_nb_n}{n}.
\ee

Using an idea from \cite{Zho2},
we introduce two sets $\Lambda$ and $\tilde{\Lambda}$ of symmetric functions,
$\tilde{\Lambda}$ for the internal edge in Figure 1,
$\Lambda$ for the external edge where the outer brane is located in Figure 1.
Hence we operators $\{\beta_n\}$ and $\{\tilde{\beta}_n\}$ acting on $\Lambda$ and $\tilde{\Lambda}$ respectively.
In physical terminology,
we introduce two systems of free bosons.

With the above preparations,
we now reformulate (\ref{eqn:ZaInG2}) in terms of the vacuum expectation value of an operator
on $\Lambda \otimes \tilde{\Lambda}$:
\ben
Z^{(a)}(\lambda;t; \bp)
& = & \langle \exp (\sum_{n=1}^\infty \frac{p_n(x)}{ni} \beta_n) q^{(a+1)K}
\sum_{\xi, \epsilon} G^\bullet_{\xi\epsilon}(\lambda;w_1,-w_1,0) \beta_{-\epsilon} \tilde{\beta}_{\xi} \\
&& \cdot \sum_{\eta} Q^{|\eta|} (-1)^{|\eta|-l(\eta)} \cdot G_{\eta}^\bullet(\lambda;-w_1,w_1, 0) \tilde{\beta}_{-\eta} \rangle.
\een
Now by definition of $G^\bullet$ and (\ref{eqn:G-1}) we have:
\ben
&& \sum_{\xi, \epsilon} G^\bullet_{\xi\epsilon}(\lambda;w_1,-w_1,0) \beta_{-\epsilon} \tilde{\beta}_{\xi}
=  \exp \sum_{n=1}^\infty \biggl( \frac{(-1)^{n-1}}{n \cdot [n]} \beta_{-n}
+ \frac{(-1)^{(n-1)}}{n \cdot [n]} \tilde{\beta}_n + \frac{1}{n} \beta_{-n} \tilde{\beta}_n \biggr), \\
&&  \sum_{\eta} Q^{|\eta|} (-1)^{|\eta|-l(\eta)} \cdot G_{\eta}^\bullet(\lambda;-w_1,w_1, 0) \tilde{\beta}_{-\eta}
= \exp \biggl( \sum_{n=1}^\infty \frac{Q^n}{n \cdot [n]} \tilde{\beta}_{-n}  \biggr),
\een
so we get:
\ben
&& Z^{(a)}(\lambda;t; \bp)
= \langle \exp (\sum_{n=1}^\infty \frac{p_n(x)}{ni} \beta_n) q^{(a+1)K} \\
&& \cdot \exp \sum_{n=1}^\infty \biggl( \frac{(-1)^{n-1}}{n \cdot [n]} \beta_{-n}
+ \frac{(-1)^{(n-1)}}{n \cdot [n]} \tilde{\beta}_n + \frac{1}{n} \beta_{-n} \tilde{\beta}_n \biggr)
\cdot \exp \biggl( \sum_{n=1}^\infty \frac{Q^n}{n \cdot [n]} \tilde{\beta}_{-n}  \biggr) \rangle \\
& = & \langle \exp (\sum_{n=1}^\infty \frac{p_n(x)}{ni} \beta_n) q^K
\exp \sum_{n=1}^\infty \biggl( \frac{(-1)^{n-1}}{n \cdot [n]} + \frac{Q^n}{n \cdot [n]} \biggr) \beta_{-n}\rangle
\cdot \exp \biggl( \sum_{n=1}^\infty \frac{(-1)^{n-1}Q^n}{n \cdot [n]^2} \biggr).
\een
When all $p_n(x) = 0$,
one gets the closed string partition function:
\be
Z(\lambda;t) := Z^{(a)}(\lambda;t;\bp)|_{p_n = 0} = \exp \biggl( \sum_{n=1}^\infty \frac{(-1)^{n-1}Q^n}{n \cdot [n]^2} \biggr).
\ee
The normalized open string amplitude is defined by:
\be
\hat{Z}^{(a)}(\lambda; t; \bp)
= \frac{Z^{(a1)}(\lambda; t; \bp)}{Z(\lambda;t)}.
\ee
To summarize,
we have the following:

\begin{proposition}
The normalized open string amplitude of the resolved conifold with one outer brane and framing $a \in \bZ$
is given by:
\be \label{eqn:ZaOperator1}
\hat{Z}^{(a)}(\lambda;t; \bp)
=  \langle \exp (\sum_{n=1}^\infty \frac{p_n(x)}{ni} \beta_n) q^{(a+1)K}
\exp \biggl( \sum_{n=1}^\infty \big( \frac{(-1)^{n-1}}{[n]} + \frac{Q^n}{[n]} \big) \frac{\beta_{-n}}{n} \biggr) \rangle.
\ee
\end{proposition}

\begin{corollary}
The normalized open string amplitude of the resolved conifold with one outer brane and framing $-1$
is given by:
\be \label{eqn:ZaOperator-1}
\hat{Z}^{(-1)}(\lambda;t; \bp)
= \exp \biggl( \sum_{n=1}^\infty \big( \frac{(-1)^{n-1}}{[n]} + \frac{Q^n}{[n]} \big) \frac{p_n(x)}{ni} \biggr).
\ee
\end{corollary}

From now on we will only consider the case when $a \neq -1$.

\subsection{Reformulation in terms of other operators}

When the framing $a \neq -1$,
we will follow the approach in \cite{Zho4} for the inner brane case to evaluate the vacuum expectation
values.
First notice that,
\ben
&& \hat{Z}^{(a)}(\lambda;t; \bp) \\
& = & \langle \exp (\sum_{n=1}^\infty \frac{p_n(x)}{ni} \beta_n) e^{i(a+1)\lambda K}
\exp \biggl( \sum_{n=1}^\infty \big( \frac{(-1)^{n-1}}{n \cdot [n]} + \frac{Q^n}{n \cdot [n]} \big) \beta_{-n} \biggr) e^{-i(a+1)\lambda K} \rangle.
 \\
& = & \langle \exp (\sum_{n=1}^\infty \frac{p_n(x)}{ni} \beta_n)
\exp \biggl( \sum_{n=1}^\infty \big( \frac{(-1)^{n-1}}{n \cdot [n]} + \frac{Q^n}{n \cdot [n]} \big)
e^{i(a+1)\lambda K}\beta_{-n}  e^{-i(a+1)\lambda K}\biggr) \rangle.
\een
By \cite[(2.14)]{Oko-Pan2},
\be
e^{uK}\beta_{-m}e^{-uK} = \cE_{-m}(um).
\ee
The operators $\cE_r(z)$ were defined in \cite{Oko-Pan1} as follows:
\be
\cE_r(z) = \sum_{k\in \bZ + \half} e^{z(k-\frac{r}{2})} E_{k-r, k} + \frac{\delta_{r,0}}{\varsigma(z)},
\ee
where the function $\varsigma(z)$ is defined by
\be
\varsigma(z) = e^{z/2} -e^{-z/2}.
\ee
The operators $\cE_r$ satisfy
\be
\cE_r(z)^* = \cE_{-r}(z)
\ee
and
\be
[\cE_a(z), \cE_b(w)] = \varsigma(aw-bz) \cE_{a+b}(z+w). \label{eqn:EE}
\ee
These operators were originally defined in the fermionic picture,
and by the boson-fermion correspondence \cite{Miw-Jim-Dat},
they also act on the bosonic Fock space,
which we take as the space $\Lambda$ of symmetric functions.
In the fermionic picture the cut-and-join operator is the operator $\cF_2$ defined by
\be
\cF_2 = \half \sum_{k\in \bZ + \frac{1}{2}} k^2 E_{kk}.
\ee

Now we can rewrite (\ref{eqn:ZaOperator1}) as follows:
\begin{multline} \label{eqn:ZaOperator2}
\hat{Z}^{(a)}(\lambda; t; \bp)
= \langle \exp \big( \sum_{n=1}^\infty \frac{p_n(x)}{i n} \beta_n\big) \\
\cdot \exp \big( \sum_{n=1}^\infty \frac{1}{n} (\frac{(-1)^{n-1}}{[n]}
+ \frac{Q^n}{[n]}) \cE_{-n}(i(a+1)n\lambda) \big) \rangle.
\end{multline}

\subsection{Operator manipulations}

For $k \neq 0$,
\be
\beta_k = \cE_k(0).
\ee
The commutation relation (\ref{eqn:EE}) specializes to
\be \label{eqn:BetaE}
[\beta_k, \cE_l(z)]= \varsigma(kz)\cE_{k+l}(z).
\ee
Starting from (\ref{eqn:BetaE}),
one can easily prove by induction:
\bea
&& \beta_k^m \cE_{-n}(z) = \sum_{j=0}^m \begin{pmatrix} m \\ j \end{pmatrix}
\varsigma(kz)^j \cE_{-n+jk}(z) \beta_k^{m-j}.
\eea
It then follows that
\bea
&& \exp (\frac{a_k}{k} \beta_k) \cdot \cE_{-n}(z)
= \sum_{j=0}^{\infty} \frac{a_k^j}{j!k^j}\varsigma(kz)^j
\cE_{-n+kj}(z) \cdot \exp (\frac{a_k}{k} \beta_k).
\eea
Repeating these formulas for all $k$,
one gets the following identity \cite{Zho4}:
\be
\exp (\sum_{k=1}^{\infty} \frac{a_k}{k} \beta_k) \cdot \cE_{-n}(z)
 \cdot \exp (\sum_{k=1}^{\infty} \frac{a_k}{k} \beta_k)
= \sum_{\mu} \frac{\tilde{a}_{\mu}(z)}{z_{\mu^+}}
\cE_{-n+|\mu|}(z),
\ee
where
\be
\tilde{a}_{\mu}(z) = \prod_k (a_k \varsigma(kz))^{m_k(\mu)}.
\ee
By (\ref{eqn:ZaOperator2}) we have
\begin{multline} \label{eqn:ZaOperator3}
\hat{Z}^{(a)}(\lambda; t; \bp)
= \langle  \exp \big( \sum_{n=1}^\infty \frac{1}{n} (\frac{(-1)^{n-1}}{[n]}
+ \frac{Q^n}{[n]})\\
\cdot  \exp \big( \sum_{m=1}^\infty \frac{p_m(x)}{i m} \beta_m\big) \cdot
\cE_{-n}(i(a+1)n\lambda) \cdot \exp \big( -\sum_{m=1}^\infty \frac{p_m(x)}{i m} \beta_m\big)\big) \rangle.
\end{multline}
Hence one can reduce the calculation of $\hat{Z}^{(a)}(\lambda;t;\bp)$ to
the computations of the correlators:
\be \label{eqn:ECorrelators}
\cor{\cE_{a_1}(b_1) \cdots \cE_{a_n}(b_n)}.
\ee
This approach was proposed for the inner brane case in \cite{Zho4}.

To evaluate the correlators (\ref{eqn:ECorrelators}),
notice that
\be
\cE_0(z) \vac = \varsigma(z) \; \vac, \qquad
\cE_n(z) \vac = 0, \qquad n > 0.
\ee
Therefore,
\be
\cor{\cE_{a_1}(b_1) \cdots \cE_{a_n}(b_n)} = 0
\ee
for $a_n > 0$ or $a_1 < 0$.
One uses (\ref{eqn:EE}) to reduce the correlators to the correlators of the form:
\be
\cor{ \cE_0(b_1) \cdots \cE_0(b_n)}
= \varsigma(b_1) \cdots \varsigma(b_n).
\ee
For example,
one can get in this fashion:

\begin{lemma}
Suppose $m_1, \dots, m_l$ is a partition of $n > 0$,
and $a_1, \dots, a_l$ are arbitrary numbers,
then the following identity holds:
\be
\cor {\beta_n \cE_{-m_1}(a_1i\lambda) \cdots \cE_{-m_l}(a_li\lambda)}
= \frac{1}{[\sum_{j=1}^l a_j]} \prod_{j=1}^l [d_j],
\ee
where $d_1 = n a_1$, and for $j > 1$,
\be
d_j = \begin{vmatrix}
n - \sum_{k=1}^{j-1} m_j & - m_j \\ \sum_{k=1}^{k-1} a_k & a_j
\end{vmatrix}.
\ee
\end{lemma}

\begin{corollary}
Suppose that $m_1, \dots, m_l$ is a partition of $n > 0$,
then the following identity holds:
\be \label{eqn:BetaCorrelator}
\cor {\beta_n \cE_{-m_1}((a+1)m_1i\lambda) \cdots \cE_{-m_l}((a+1)m_li\lambda)}
= \frac{1}{[(a+1)n]} \prod_{j=1}^l [(a+1)n m_j].
\ee
\end{corollary}

\subsection{The open string free energy and the $n$-point functions}
Write
\be
F^{(a)}(\lambda; t; \bp)
: = \log \hat{Z}^{(a)}(\lambda; t; \bp).
\ee
It will be referred to as the open string free energy of the resolved conifold with frame $a$.
Write
\be
F^{(a)}(\lambda; t; \bp) = \sum_{n=1}^\infty \frac{1}{n!}
\sum_{m_1, \dots, m_n \geq 1}  F^{(a)}_{m_1, \dots, m_n} (\lambda; t) p_{m_1} \cdots p_{m_n},
\ee
where the coefficients $F^{(a)}_{m_1, \dots, m_n} (\lambda; t)$,
which are symmetric in $m_1, \dots, m_n$,
are called the $n$-point functions.
Write
\be
F^{(a)}_{m_1, \dots, m_n} (\lambda; t) = \sum_{g=0}^\infty \lambda^{2g-2+n} F_{g; m_1, \dots, m_n}^{(a)}(t),
\ee
and
\be
\Phi^{(a)}_{g,n}(\bp) : =  \frac{1}{n!} \sum_{m_1, \dots, m_n \geq 1}F_{g; m_1, \dots, m_n}^{(a)}(t)
p_{m_1} \cdots p_{m_n}.
\ee
The symmetrization of $\Phi^{(a)}_{g,n}(\bp)$ is defined by:
\be
\Psi^{(a)}_{g,n}(t; x_1, \dots, x_n) : =  \sum_{m_1, \dots, m_n \geq 1}F_{g; m_1, \dots, m_n}^{(a)}(t)
x_1^{m_1} \cdots x_n^{m_n}.
\ee
We also set
\be
\Psi^{(a)}_n(\lambda; t; x_1, \dots, x_n)
= \sum_{g \geq 0} \lambda^{2g-2+n} \Psi^{(a)}_{g,n}(t; x_1, \dots, x_n).
\ee

\section{One-Point Functions}

In this section we will use the method in last section to find closed formula for one-point functions
for open string invariants.

\subsection{Preliminary results for one-point functions}

It is easy to see that the one-point functions are given by:
\ben
&& \sum_{n=1}^\infty F^{(a)}_n(\lambda; t) p_n(x) \\
& = & \lvac \sum_{n=1}^\infty \frac{p_n(x)}{n i} \beta_n \cdot
\exp \big( \sum_{n=1}^\infty \frac{1}{n} (\frac{(-1)^{n-1}}{[n]}
+ \frac{Q^n}{[n]}) \cE_{-n}(i(a+1)n\lambda) \big) \vac \\
& = & \sum_{n=1}^\infty \frac{p_n(x)}{n i}
\sum_{m_1 \cdot 1 + \cdots + m_l \cdot l = n} \prod_{j=1}^l
\frac{(\frac{(-1)^{j-1}}{[j]} + \frac{Q^j}{[j]})^{m_j}}{j^{m_j}m_j!} \cdot
\cor{ \beta_n \prod_{j=1}^l \cE_{-j}(i(a+1)j\lambda)^{m_j}}. \
\een
Now by (\ref{eqn:BetaCorrelator}),
we get:
\begin{multline} \label{eqn:F(x1)}
\sum_{n \geq 1} F_n^{(a)}(\lambda; t) p_n(x) \\
= \sum_{n=1}^\infty \frac{p_n(x)}{n i}
\sum_{m_1 \cdot 1 + \cdots + m_l \cdot l = n} \prod_{j=1}^l
\frac{(\frac{(-1)^{j-1}}{[j]} + \frac{Q^j}{[j]})^{m_j}}{j^{m_j}m_j!} \cdot
\frac{1}{[(a+1)n]} \prod_{j=1}^l [(a+1)jn]^{m_j}.
\end{multline}
After symmetrization we get:
\begin{multline}
\Psi^{(a)}_1(\lambda; t; x)
= \sum_{n=1}^\infty \frac{x^n}{n i \cdot [(a+1)n]} \\
\cdot \sum_{\sum_{j=1}^l j m_j = n} \prod_{j=1}^l
\frac{\big((\frac{(-1)^{j-1}}{[j]} + \frac{Q^j}{[j]}) \cdot  [(a+1)jn] \big)^{m_j}}{j^{m_j}m_j!}.
\end{multline}

The following are the first few terms:
\ben
&& F^{(a)}_1(\lambda; t) = \frac{1}{[1]}+ \frac{Q}{[1]}.
\een
\ben
F^{(a)}_2(\lambda; t)
& = &  \frac{1}{4} ( - \frac{1}{[2]}+ \frac{Q^2}{[2]} ) \frac{[4(a+1)]}{[2(a+1)]}
+ \frac{1}{4} ( \frac{1}{[1]}+ \frac{Q}{[1]} )^2 [2(a+1)].
\een
\begin{multline*}
F^{(a)}_3(\lambda; t) = \frac{1}{3} \cdot ( \frac{1}{3[3]} + \frac{Q^3}{3[3]} ) \frac{[9(a+1)]}{[3(a+1)]} \\
+ \frac{1}{3} \cdot \frac{1}{2} ( \frac{1}{[1]}+ \frac{Q}{[1]} )
( - \frac{1}{[2]}+ \frac{Q^2}{[2]}) \frac{[6(a+1)][3(a+1)]}{[3(a+1)]}
+ \frac{1}{3} \cdot \frac{1}{3!} ( \frac{1}{[1]}+ \frac{Q}{[1]} )^3 \frac{[3(a+1)]^3}{[3(a+1)]}.
\end{multline*}
In particular,
when $a=0$,
we have
\ben
&& F^{(0)}_1(\lambda; t) = \frac{1}{[1]}+ \frac{Q}{[1]}, \\
&& F^{(0)}_2(\lambda;t) = \frac{1}{2[2]} + \frac{[2]}{2[1]^2} Q + \frac{[3]}{2[1][2]}Q^2, \\
&& F^{(0)}_3(\lambda; t) = \frac{1}{3[3]} + \frac{[3]}{3 [1]^2} Q + \frac{[3][4]}{3[1]^2[2]}Q^2
+ \frac{[4][5]}{3[1][2][3]}Q^3.
\een
We will present below a method to take the summations on the right-hand side of (\ref{eqn:F(x1)}).
It will provide us explicit expressions for genera $F^{(a)}_n(\lambda;t)$.

\subsection{The  genus $0$ case} \label{sec:g=0}
To motivate our method,
we will first treat the $g=0$ case,
where our method is elementary.
By (\ref{eqn:F(x1)}),
we get:
\be \label{eqn:Psi(g=0)}
\Psi_{0,1}^{(a)}(t; x)
= -  \sum_{n=1}^\infty \frac{x^n}{n^2(a+1)}
\sum_{m_1 \cdot 1 + \cdots + m_l \cdot l = n} \prod_{j=1}^l
\frac{ [n(a+1)((-1)^{j-1} + Q^j) ]^{m_j}}{j^{m_j}m_j!}.
\ee
The following are the first few terms:
\ben
-\Psi^{(a)}_{0,1}(t; x)  & = & (1 + Q) x + \frac{1}{4} ((2a+1)+4(a+1)Q+(2a+3)Q^2) x^2 \\
& + & \frac{1}{9} (1+\frac{9a(a+1)}{2} + \frac{9(a+1)(3a+2)}{2}Q+\frac{9(a+1)(3a+4)}{2}Q^2 \\
&& +(1 + \frac{9(a+1)(a+2)}{2} Q^3)x^3  + \cdots.
\een
In particular,
when $a= 0$,
\ben
-\Psi^{(0)}_{0,1}(t; x)  & = & (1 +Q) x + \frac{1}{4} (1+4Q+3Q^2) x^2
+ \frac{1}{9} (1+9Q+18Q^2+10Q^3)x^3 \\
& + & \frac{1}{16} (1+16Q+60Q^2+80Q^3+35Q^4)x^4 \\
& + & \frac{1}{25} (1 + 25Q+ 150Q^2 + 350Q^3 + 350Q^4 + 126Q^5)x^5 + \cdots.
\een

To evaluate the summation
\be \label{eqn:Sum}
\sum_{m_1 \cdot 1 + \cdots + m_l \cdot l = n} \prod_{j=1}^l
\frac{ [n(a+1)((-1)^{j-1} + Q^j ) ]^{m_j}}{j^{m_j}m_j!},
\ee
we regard it as the coefficient of $z^n$ of the following series:
\be \label{eqn:Series}
\sum_{N=0}^\infty z^N \sum_{m_1 \cdot 1 + \cdots + m_l \cdot l = N} \prod_{j=1}^l
\frac{ [n(a+1)((-1)^{j-1} + Q^j ) ]^{m_j}}{j^{m_j}m_j!}.
\ee
This can be easily rewritten as follows:
\be
\exp \sum_{j=1}^\infty \frac{n(a+1)((-1)^{j-1} + Q^j )}{j} z^j
= \frac{(1+z)^{n(a+1)}}{(1 - Qz)^{n(a+1)}}.
\ee
Now we have the binomial expansions:
\be
\frac{1}{(1-Qz)^{n(a+1)}} = \sum_{j=0}^\infty \frac{\prod_{k=0}^{j-1} (n(a+1)+k)}{j!} Q^j z^j
\ee
and
\be
(1+z)^{n(a+1)} = \sum_{m \geq 0} \frac{\prod_{k=0}^{m-1} (n(a+1)-k)}{m!} z^m,
\ee
it follows that the coefficient of $z^n$ of the series in (\ref{eqn:Series}) is
\ben
&& \sum_{j=0}^n  \frac{\prod_{k=0}^{j-1} (n(a+1)+k)}{j!} Q^j  \cdot
\frac{\prod_{k=0}^{n-j-1} (n(a+1)-k)}{(n-j)!} \\
& = & n(a+1) \cdot \sum_{j=0}^n  \frac{\prod_{k=1}^{n-1} (na+j+k)}{j!(n-j)!} Q^j.
\een
So we have proved our first main result:

\begin{theorem} \label{thm:g=0}
For the resolved confiold we have
\be \label{eqn:Psi(0,1)}
\Psi_{0,1}^{(a)}(t; x)
= -  \sum_{n=1}^\infty \frac{x^n}{n}  \sum_{j=0}^n  \frac{\prod_{k=1}^{n-1} (na+j+k)}{j!(n-j)!} Q^j.
\ee
\end{theorem}

\subsection{The case of arbitrary genera} \label{sec:g>0}

Similarly,
we understand the summation
\be \label{eqn:Sum(g)}
\sum_{\sum_{j=1}^l j m_j = n} \prod_{j=1}^l
\frac{\big((\frac{(-1)^{j-1}}{[j]} + \frac{Q^j}{[j]}) \cdot  [(a+1)jn] \big)^{m_j}}{j^{m_j}m_j!}
\ee
in (\ref{eqn:F(x1)}) as the coefficient of $z^n$ in
\ben
f^{(a)}_n(z; t) & = & \sum_{N\geq 0} z^N \sum_{m_1 \cdot 1 + \cdots + m_l \cdot l = N} \prod_{j=1}^l
(\frac{(\frac{(-1)^{j-1}}{[j]} + \frac{Q^j}{[j]})^{m_j}}{j^{m_j}m_j!}) \cdot
 \prod_{j=1}^l [(a+1)jn]^{m_j} \\
& = & \exp \sum_{j=1}^{\infty} \frac{\frac{(-1)^{j-1}}{[j]} + \frac{Q^j}{[j]}}{j} [(a+1)jn] z^j.
\een
Because
\be
\frac{[j(a+1)n]}{[j]} = \sum_{k=1}^{(a+1)n} q^{j((a+1)n-2k+1)/2}
\ee
we have
\ben
&& f^{(a)}_n(z;t) \\
& = & \exp \sum_{k=1}^{(a+1)n} \sum_{j=1}^\infty \frac{(-1)^{j-1}}{j}
(q^{j((a+1)n-2k+1)/2} - (-1)^j Q^j q^{j((a+1)n-2k+1)/2})z^j \\
& = & \prod_{k=1}^{(a+1)n} \frac{1 + q^{((a+1)n-2k+1)/2}z}{1-Qq^{((a+1)n-2k+1)/2}z}.
\een
We have the quantum binomial formula:
\be
\prod_{k=1}^{(a+1)n} (1 + q^{((a+1)n-2k+1)/2}z)
= \sum_{j=0}^{(a+1)n} \begin{bmatrix} (a+1)n \\ j \end{bmatrix} z^j,
\ee
and
\be
\prod_{k=1}^{(a+1)n} \frac{1}{1-Qq^{((a+1)n-2k+1)/2}z}
= \sum_{j=0}^\infty \begin{bmatrix} (a+1)n+j-1 \\ j \end{bmatrix} Q^jz^j,
\ee
where the quantum binomial coefficients are defined by:
\be
\begin{bmatrix} n \\ j \end{bmatrix}
= \frac{[n][n-1] \cdots [n-j+1]}{[j][j-1] \cdots [1]}.
\ee
Indeed, using the definition of elementary symmetric functions one has
\ben
&& \prod_{k=1}^{(a+1)n} (1 + q^{((a+1)n-2k+1)/2}z) \\
& = & \sum_{j=0}^n e_j(1, q \cdots, q^{(a+1)n-1}) q^{-j((a+1)n-1)/2} z^j
= \sum_{j=0}^{(a+1)n} \begin{bmatrix} (a+1)n \\ j \end{bmatrix} z^j,
\een
where in the second equality we have used the following equality \cite[p. 26]{Mac}:
\ben
e_j(1, q, \dots, q^{(a+1)n-1})
= q^{j(j-1)/2} \frac{(1-q^{(a+1)n})(1-q^{(a+1)n-1}) \cdots (1 - q^{(a+1)n-j+1})}{(1-q)(1-q^2) \cdots(1-q^j)};
\een
Similarly,
using the definition of complete symmetric functions  one has
\ben
&& \prod_{k=1}^{(a+1)n} \frac{1}{1-Qq^{((a+1)n-2k+1)/2}x} \\
& = & \sum_{j=0}^\infty h_j(1, q, \dots, q^{(a+1)n-1}) q^{-j((a+1)n-1)/2}Q^j z^j \\
& = & \sum_{j=0}^\infty \begin{bmatrix} (a+1)n+j-1 \\ j \end{bmatrix} Q^j z^j,
\een
where in the second equality we have used the following equality \cite[p. 26]{Mac}:
\ben
h_j(1, q, \dots, q^{(a+1)n-1})
= \frac{(1-q^{(a+1)n+j-1})(1-q^{(a+1)n+j-2}) \cdots (1 - q^{(a+1)n})}{(1-q)(1-q^2) \cdots (1-q^j)}.
\een

It follows that we have
\ben
&& f^{(a)}_n(z;t)
=  \sum_{j=0}^\infty \begin{bmatrix} (a+1)n+j-1 \\ j \end{bmatrix} Q^j z^j
\cdot \sum_{j=0}^{(a+1)n} \begin{bmatrix} (a+1)n \\ j \end{bmatrix} z^j,
\een
hence its coefficient of $z^n$ is:
\be
\sum_{j=0}^n \begin{bmatrix} (a+1)n+j-1 \\ j \end{bmatrix} \cdot
\begin{bmatrix} (a+1)n \\ n-j \end{bmatrix} Q^j
= [(a+1)n] \cdot \sum_{j=0}^n \frac{\prod_{k=1}^{n-1}[an+j+k]}{[j]![n-j]!} Q^j.
\ee
Therefore,
the following formula for $a \geq 0$:
\be
\Psi_{1}^{(a)}(\lambda;t; x)
=\sum_{n=1}^\infty \frac{x^n}{ni}  \sum_{j=0}^{n}  \frac{\prod_{k=1}^{n-1}[an+j+k]}{[j]! [n-j]!} Q^j.
\ee
This also holds when $a+1 = - b \leq -1$.
In this case $f^{(a)}_n(z; t)$ is equal to:
\ben
&& \sum_{N\geq 0} z^N \sum_{m_1 \cdot 1 + \cdots + m_l \cdot l = N} \prod_{j=1}^l
(\frac{(\frac{(-1)^{j-1}}{[j]} + \frac{Q^j}{[j]})^{m_j}}{j^{m_j}m_j!}) \cdot
 \prod_{j=1}^l [-bjn]^{m_j} \\
& = & \exp \sum_{j=1}^{\infty} - \frac{\frac{(-1)^{j-1}}{[j]} + \frac{Q^j}{[j]}}{j} [bjn] z^j \\
& = & \exp - \sum_{k=1}^{bn} \sum_{j=1}^\infty \frac{(-1)^{j-1}}{j}
(q^{j(bn-2k+1)/2}  - (-1)^j Q^j q^{j(bn-2k+1)/2}) z^j \\
& = & \prod_{k=1}^{bn} \frac{1-Qq^{(bn-2k+1)/2}x}{1 + q^{(bn-2k+1)/2} z}
= \sum_{j=0}^{bn} \begin{bmatrix} bn \\ j \end{bmatrix} (-Qz)^j \cdot
\sum_{j=0}^\infty \begin{bmatrix} bn+j-1 \\ j \end{bmatrix} (-z)^j.
\een
Its coefficient of $z^n$ is:
\ben
&& \sum_{j=0}^n \begin{bmatrix} bn \\ j \end{bmatrix} (-Q)^j
\cdot \begin{bmatrix} bn+n-j-1 \\ n-j \end{bmatrix} (-1)^{n-j} \\
& = & (-1)^n \frac{[bn] \cdot \prod_{k=1}^{n-1} [bn-j + (n-k)]}{[j]!\cdot [n-j]!]} \\
& = &  \frac{[(a+1)n] \cdot  \prod_{k=1}^{n-1} [an+j +k]}{[j]!\cdot [n-j]!]}.
\een
So we have proved the following generalization of Theorem \ref{thm:g=0}:

\begin{theorem} \label{thm:g}
For the resolved confiold we have
\be \label{eqn:Psi(g,1)}
\Psi_{1}^{(a)}(\lambda; t; x)
= \sum_{n=1}^\infty \frac{x^n}{ni}  \sum_{j=0}^{n}  \frac{\prod_{k=1}^{n-1} [an+j+k]}{[j]! [n-j]!} Q^j.
\ee
\end{theorem}

\section{Ooguri-Vafa Integrality for One-Point Functions}

In this section we check the Ooguri-Vafa Integrality Conjecture for the one-point function
of the resolved conifold in general framing.
This recovers and generalizes some results by Aganagic-Vafa \cite{AV}, Aganagic-Klemm-Vaf \cite{AKV}  and
Mari\~no-Vafa \cite{Mar-Vaf}.

\subsection{Genus $0$  Ooguri-Vafa integral invariants}

Ooguri and Vafa \cite{Oog-Vaf} conjectured
an integral property of open string invariant of Calabi-Yau $3$-folds,
similar to the Gopakumar-Vafa integrality \cite{Gop-Vaf2} for closed string invariants.
For the resolved conifold with one outer brane,
it takes the following form for genus $0$ one-point functions \cite{AV}:
\be \label{eqn:OV}
x\frac{\pd}{\pd x} \Psi^{(a)}_{0,1}(t; x) = \sum_{m,n \geq 1} m d^{(a)}_{k,m} \log (1 - e^{-kt} x^m)
\ee
for some integers $d^{(a)}_{k,m}= \sum_s N^{(a)}_{k,m,s}$.
Aganagic and Vafa \cite{AV} checked this for $a=0$,
and a table for $d_{k,m}:=d^{(0)}_{k,m,0}$ can be found in their paper.
We now establish the general case.

By (\ref{eqn:Psi(0,1)}),
\be \label{eqn:DPsi}
x\frac{\pd}{\pd x} \Psi_{0,1}^{(a)}(t; x)
= -  \sum_{n=1}^\infty x^n \sum_{k=0}^n  \frac{\prod_{j=1}^{n-1} (na+j+k)}{k!(n-k)!} (-1)^ke^{-kt}.
\ee
The right-hand side of (\ref{eqn:OV}) can be written as:
\be
-  \sum_{k=0}^{\infty}\sum_{m=1}^{\infty} m d^{(a)}_{k,m} \sum_{j=1}^{\infty} \frac{1}{j} e^{-jkt} x^{jm}.
\ee
The first few terms are£º
\begin{multline*}
- d^{(a)}_{0,1} (x+ \frac{1}{2} x^2 + \frac{1}{3} x^3 + \cdots)
- 2d^{(a)}_{0,2} (x^2 + \frac{1}{2} x^4 + \frac{1}{3} x^6 + \cdots ) \\
- 3d^{(a)}_{0,3} (x^3 + \frac{1}{2} x^6 + \frac{1}{3} x^9 + \cdots )
 - \cdots \cdots \\
- d^{(a)}_{1,1} (e^{-t} x+ \frac{1}{2} e^{-2t} x^2 + \frac{1}{3} e^{-3t} x^3 + \cdots)
- 2d^{(a)}_{1,2} (e^{-t} x^2 + \frac{1}{2} e^{-2t} x^4 + \frac{1}{3} e^{-3t} x^6 + \cdots ) \\
- 3d^{(a)}_{1,3} (e^{-t} x^3 + \frac{1}{2} e^{-2t} x^6 + \frac{1}{3} e^{-3t} x^9 + \cdots )
- \cdots \cdots \\
- d^{(a)}_{2,1} (e^{-2t} x+ \frac{1}{2} e^{-4t} x^2 + \frac{1}{3} e^{-6t} x^3 + \cdots)
- 2d^{(a)}_{2,2} (e^{-2t} x^2 + \frac{1}{2} e^{-4t} x^4 + \frac{1}{3} e^{-6t} x^6 + \cdots ) \\
- 3d^{(a)}_{2,3} (e^{-2t} x^3 + \frac{1}{2} e^{-4t} x^6 + \frac{1}{3} e^{-6t} x^9 + \cdots )
- \cdots \cdots \\
\end{multline*}
Putting similar terms together£º
\begin{multline*}
- d^{(a)}_{0,1} x - (2 d^{(a)}_{0,2} + \frac{1}{2} d^{(a)}_{0,1})x^2
- (3d^{(a)}_{0,3} + \frac{1}{3} d^{(a)}_{0,1}) x^3
- (4d^{(a)}_{0,4} + d^{(a)}_{0,2} + \frac{1}{4} d^{(a)}_{0,1}) x^4  - \cdots \\
- d^{(a)}_{1,1} e^{-t} x - 2 d^{(a)}_{1,2} e^{-t} x^2
- 3 d^{(a)}_{1,3} e^{-t} x^3 - 4 d^{(a)}_{1,4} e^{-t} x^4 - \cdots \\
- d^{(a)}_{2,1} e^{-2t} x - (2d^{(a)}_{2,2} + \frac{1}{2} d^{(a)}_{1,1}) e^{-2t} x^2 - 3 d^{(a)}_{2,3} e^{-2t} x^3
- (4d^{(a)}_{2,4} + d^{(a)}_{1,2}) e^{-2t} x^4 - \cdots \\
- d^{(a)}_{3,1} e^{-3t} x - 2d^{(a)}_{3,2} e^{-3t} x^2 - (3d^{(a)}_{3,3} + \frac{1}{3} d^{(a)}_{1,1}) e^{-3t} x^3
- d^{(a)}_{3,4} e^{-3t} x^4 - \cdots.
\end{multline*}
The coefficient of $e^{-kt} x^m$ is of the form:
\be
-m d^{(a)}_{k,m} - \sum_{j|(k,m), j > 1} \frac{m}{j^2} d^{(a)}_{k/j, m/j}.
\ee
Therefore,
by comparing with the right-hand side of (\ref{eqn:DPsi}):
\be \label{eqn:d(k,m)Recursion}
d^{(a)}_{k,m} + \sum_{j|(k,m), j > 1} \frac{1}{j^2} d^{(a)}_{k/j, m/j}
= (-1)^k \frac{\prod_{j=1}^{m-1} (ma+j+k)}{m \cdot k!(m-k)!}.
\ee
From this one can recursively determine all $d^{(a)}_{k,m}$.
Let $\mu: \bN \to \{0,1, -1\}$ be the M\"obius mu functions.
I.e,
\be
\mu(n) = \begin{cases}
(-1)^r, & n = p_1 \cdots p_r, p_1, \dots, p_r \;\; \text{are distinct primes}, \\
0, & \text{otherwise}.
\end{cases}
\ee
Then
\be \label{eqn:d(k,m)Solution}
d_{k,m}^{(a)}
= \sum_{n | (k,m)}  \frac{\mu(n)}{n^2} (-1)^{k/n} \frac{\prod_{j=1}^{m/n-1}
(ma/n+j+k/n)}{(m/n) \cdot (k/n)!(m/n-k/n)!}.
\ee
Indeed,
(\ref{eqn:d(k,m)Recursion}) has a unique solution,
and by the following property of the M\"bius mu function
\be
\sum_{l|n} \mu(l) = \delta_{n,1},
\ee
one can check that (\ref{eqn:d(k,m)Solution}) is a solution.

For example,
when $(k, m) = 1$,
\be
d^{(a)}_{k,m} = \frac{(-1)^k}{m} \frac{\prod_{j=1}^{m-1} (ma+j+k)}{k! \cdot (m-k)!}, \qquad m > 1.
\ee
If $k$ is a prime number $p$,
$$d^{(a)}_{p,pn} = \frac{(-1)^p}{m} \frac{\prod_{j={1}}^{m-1} (ma+j+p)}{p! \cdot (pn-p)!}
- \frac{1}{p^2} d_{1,n}^{(a)},$$
therefore,
\be
d^{(a)}_{p,pn} = \frac{(-1)^p}{m} \frac{\prod_{j={1}}^{m-1} (ma+j+p)}{p! \cdot (pn-p)!}
+ \frac{\prod_{j=2}^n (na+j)}{p^2 \cdot n!}.
\ee
If $k$ is a square of a prime $p$, $m=pn$, $(k,n) =1$,
then
\be
d^{(a)}_{p^2,pn} =  \frac{(-1)^p}{pn} \frac{\prod_{j=1}^{pn-1} (pna+j+p^2)}{p^2! \cdot (pn-p^2)!}
+(-1)^{p-1}\frac{\prod_{i=1}^{n-1} (na+i+p) }{p^2n \cdot p!(n-p)!}.
\ee

When $a$ is taken to be $0$,
the numbers $d^{(0)}_{k,m}$ match with the numbers $d_{k,m}$ in Table 1 in \cite{AV}.
The above formula for $d^{(a)}_{k,m}$ now takes the following simplified forms:
\be
d_{k,m}
= \sum_{n | (k,m)}  \frac{\mu(n)}{n^2} (-1)^{k/n} \frac{(m/n+k/n-1)!}{(m/n) \cdot (k/n)!^2(m/n-k/n)!}.
\ee
When $(k,m) = 1$,
\be \label{eqn:d(k,m)}
d_{k,m} = \frac{(-1)^k}{m} \frac{(m+k-1)!}{k! \cdot k! \cdot (m-k)!}.
\ee
If $k$ is a prime number $p$,
\be
d_{p,pn} =\frac{(-1)^p}{pn} \frac{(pn+p-1)!}{p! \cdot p! \cdot (pn-p)!} + \frac{1}{p^2}.
\ee
If $k$ is a square of a prime $p$, $m=pn$ and $(p,n) =1$,
then
\be
d_{p^2,pn} =  \frac{(-1)^p}{pn} \frac{(pn+p^2-1)!}{p^2! \cdot p^2! \cdot (pn-p^2)!}
+ (-1)^{p-1} \frac{(n+p-1)! }{p^2n \cdot p! \cdot p!(n-p)!}.
\ee

It is easy to see that
\ben
d_{m,m}
& = & \sum_{n | m}  \frac{\mu(n)}{n^2} (-1)^{m/n} \frac{(2m/n-1)!}{(m/n) \cdot (m/n)!^2(m/n-m/n)!} \\
& = & \frac{1}{2m^2} \sum_{l|m} (-1)^l \mu(m/l) \binom{2l}{l}.
\een
The first few terms of the sequence $(|d_{m,m}|)_{m \geq 1}$  are given by
\ben
&& 1, 1, 1, 2, 5, 13, 100, 300, 925, 2911, 9386, 30771, 102347, 344705, 1173960, \dots,
\een
matches with the sequence number A131868 of the AT\&T On-Line Encyclopedia of Integer Sequences.
It has the following combinatorial meaning:
$na(n)$ is the number of $n$-member subsets of $\{1,2,3,...,2n-1\}$ that sum to $1 \pmod n$.
Similarly,
\be
d_{k,k+1}
= (-1)^{k} \frac{(2k)!}{(k+1) \cdot k!^2}.
\ee
The first few terms  of the sequence $(|d_{k,k+1}|)_{k \geq 1}$ are given by
\ben
&& 1, 1, 1, 2, 5, 14, 42, 132, 429, 1430, 4862, 16796, 58786, 208012, 742900, 267440, \\
&& 9694845, 35357670, \dots,
\een
This is the sequence A000108 which are the Catalan numbers (also called Segner numbers):
\be
C(n) = \frac{1}{n+1}\binom{2n}{n} = \frac{(2n)!}{n!(n+1)!}.
\ee
These numbers have numerous combinatorial meanings.
It will be interesting to find out combinatorial meanings for the sequences $(d_{k, k+n})_{k \geq 1}$
for other fixed $n$.

\subsection{Another related half integral property}

One can also do the following expansion:
\ben
&& x\frac{\pd}{\pd x} \Psi^{(a)}_{0,1}(t; x)  = \log (1 - x)
+ a\log (1 -x^2) + \frac{3a(a+1)}{2} \log(1-x^3) \\
& + & \frac{4a(a+1)(2a+1)}{3} \log(1-x^4)
+ \frac{25a(a+1)(5a^2+5a+2)}{24} \log(1 -x^5) \\
& + & \frac{3a(a+1)(36a^3+54a^2+31a+4)}{10} \log(1-x^6)  \\
& + & \log(1-Qx) + 2(a+1) \log(1-Qx^2) + \cdots
\een
This also has some integral properties.

\begin{conjecture}
There are half integers $e^{(a)}_{k,m}$ such that
\be
x\frac{\pd}{\pd x} \Psi^{(a)}_{0,1}(t; x) = \sum_{m,n \geq 1} m e^{(a)}_{k,m} \log (1 - Q^k x^m)
\ee
where $Q=-e^{-t}$.
\end{conjecture}

Similar to (\ref{eqn:d(k,m)Recursion}) we have
\be \label{eqn:e(k,m)Recursion}
m e^{(a)}_{k,m} + \sum_{j|(k,m), j > 1} \frac{m}{j^2} e^{(a)}_{k/j, m/j}
= \frac{\prod_{j=1}^{m-1} (ma+j+k)}{k!(m-k)!}.
\ee
Then
\be
e_{k,m}^{(a)}
= \sum_{n | (k,m)}  \frac{\mu(n)}{n^2} \frac{\prod_{j=1}^{m/n-1}
(ma/n+j+k/n)}{(m/n) \cdot (k/n)!(m/n-k/n)!}.
\ee
For example, when $(k, m) = 1$,
\be
e^{(a)}_{k,m} = \frac{1}{m} \frac{\prod_{j={k+1}}^{m+k-1} (ma+j)}{k! \cdot (m-k)!}, \qquad m > 1,
\ee
in particular, when $a=0$,
\be
e^{(0)}_{k,m} = \frac{1}{m} \frac{(m+k-1)!}{k! \cdot k! \cdot (m-k)!}.
\ee
Clearly,
\be
e^{(a)}_{k,m} = |d^{(a)}_{k,m}|
\ee
for $(k,m) =1$.
For $(k,m) \neq 1$,
they also coincide for many cases.
Take $a=0$ case for example.
We have the following table for $e_{k,m}:=e^{(0)}_{k,m}$:

\begin{center}
\begin{tabular}{|c|cccccc|}
\hline
$m$ & $e_{1,m}$ & $e_{2,m}$ & $e_{3,m}$ & $e_{4,m}$ & $e_{5,m}$ & $e_{6,m}$ \\
\hline
 $1$ &  $1$ & 0 & 0 & 0 & 0 & 0 \\
 $2$ &  $1$ & \fbox{1/2} & 0 & 0 & 0 & 0 \\
 $3$ &  $1$ & 2 & 1 & 0 & 0 & 0 \\
 $4$ &  $1$ & \fbox{7/2} & 5 & 2 & 0 & 0 \\
 $5$ &  $1$ & 6 & 14 & 14 & 5 & 0 \\
 $6$ &  $1$ & \fbox{17/2} & 31 & 52 & 42 & \fbox{25/2} \\
 $7$ &  $1$ & 12 & 60 & 150 & 198 & 132 \\
 $8$ &  $1$ & \fbox{31/2} & 105 & 360 & 693 & \fbox{1499/2} \\
 $9$ &  $1$ & 20 & 171 & 770 & 2002 & 3114 \\
$10$ &  $1$ & \fbox{49/2} & 264 & 1500 & 5045 & \fbox{10507} \\
$11$ &  $1$ & 30 & 390 & 2730 & 11466 & 30576 \\
$12$ &  $1$ & \fbox{71/2} & 556 & 4690 & 24024 & \fbox{158809/2} \\
$13$ &  $1$ & 42 & 770 & 7700 & 47124 & 188496 \\
$14$ &  $1$ & \fbox{97/2} & 1040 & 12152 & 87516 & \fbox{415686} \\
$15$ &  $1$ & 56 & 1375 & 18564 & 155195 & 862194 \\
$16$ &  $1$ & \fbox{127/2}\; & 1785 & 27552 & 264537 & \fbox{3394839/2} \\
\hline
\end{tabular}
\end{center}

Comparing with the table of $d_{k,m}$ in \cite{AV},
we notice that $e_{k,m} = |d_{k,m}|$ for most of the pairs $(k,m)$.
(The places where they are different are put in boxes.)

\subsection{Some integral invariants in arbitrary genera}

By the integrality prediction of Ooguri-Vafa \cite{Oog-Vaf},
there should be integers $N^{(a)}_{m,k,s}$ such that
\be \Psi^{(a)}_1(\lambda; t; x)
= \sum_{n=1}^\infty \sum_{m, k, s} \frac{N_{m, k, s}}{in[n]} e^{n(-kt + is \lambda)} ((-a)^ax)^{nm},
\ee
where $m,k \in \bZ$, $s \in \half \bZ$.
By Theorem \ref{thm:g},
we should have:
\be \label{eqn:OVg}
\sum_{n=1}^\infty \sum_{m, k, s} \frac{N^{(a)}_{m, k, s}}{i[n]} e^{n(-kt + is \lambda)} x^{nm}
= \sum_{n=1}^\infty \frac{((-1)^ax)^n}{i}  \sum_{j=0}^{n}  \frac{\prod_{k=1}^{n-1} [an+j+k]}{[j]! [n-j]!} Q^j.
\ee

For fixed $m, k$,
$N^{(a)}_{m,k,s} = 0$ for $s\gg 0$,
so that
$$N_{m,k}^{(a)}(q^{1/2}) = \sum_{s \in \half \bZ} N^{(a)}_{k,m,s} q^s
\in \bZ[q^{1/2}, q^{-1/2}]$$
is a Laurent polynomial in $q^{1/2}$.
Comparing the coefficients of $x^m e^{-kt}$ on both sides of (\ref{eqn:OVg}),
one gets:
\be
\sum_{j | (k,m)} \frac{N^{(a)}_{m/j, k/j}(q^{j/2})}{[j]}
= (-1)^{ma+k} \frac{\prod_{j=1}^{m-1} [am+j+k]}{[k]! [m-k]!}.
\ee
Hence
\be
\frac{N^{(a)}_{m,k}(q^{1/2})}{[1]} = \sum_{n |(k,m)} \mu(n) (-1)^{(ma+k)/n}
\frac{\prod_{j=1}^{m/n-1} [am+nj+k]}{\prod_{j=1}^{k/n} [nj]\cdot \prod_{j=1}^{(m-k)/n} [nj]}.
\ee
In particular,
when $(m,k) = 1$,
\be
N^{(a)}_{m,k}(q^{1/2}) = (-1)^{ma+k}[1] \frac{\prod_{j=1}^{m-1} [am+j+k]}{[k]! [m-k]!}.
\ee
When $k$ is a prime $p$ and $m = kn$ for some integer $n$,
\be
N^{(a)}_{pn,p}
= (-1)^{pna+p}[1] \frac{\prod_{j=1}^{pn-1} [apn+j+p]}{[p]! [pn-p]!}
- (-1)^{na+1}[1] \frac{\prod_{j=1}^{n-1} [an+pj+1]}{[p] \cdot \prod_{j=1}^{n-1} [pj]!}.
\ee
When $k=p^2$ for some prime, $m=pn$ and $(p,n) =1$,
\be
N^{(a)}_{pn,p^2}
= (-1)^{pna+p^2}[1] \frac{\prod_{j=1}^{pn-1} [apn+j+p^2]}{[p^2]! [pn-p^2]!}
-(-1)^{na+p} [1] \frac{\prod_{j=1}^{n-1} [anp+pj+p^2]}{\prod_{j=1}^{p} [pj]\cdot \prod_{j=1}^{n-p} [pj]}.
\ee

For example,
when $a =0$,
$i$ times the right-hand side of (\ref{eqn:OVg}) is:
\ben
&& (\frac{1}{[1]} - \frac{1}{[1]}e^{-t} ) x
+ (\frac{1}{[2]}- \frac{[2]}{[1]^2}e^{-t}  + \frac{[3]}{[2][1]}e^{-2t}) x^2\\
& + & (\frac{1}{[3]} - \frac{[3]}{[1]^2} e^{-t}
+ \frac{[3][4]}{[2][1]^2}e^{-2t} - \frac{[4][5]}{[3][2][1]}e^{-3t})x^3 + \cdots \\
& = & (\frac{1}{[1]} x + \frac{1}{[2]} x^2 + \frac{1}{[3]} x^3 + \cdots) \\
& - & (\frac{1}{[1]} xe^{-t} + \frac{1}{[2]} x^2e^{-2t} + \frac{1}{[3]} x^3 e^{-3t} + \cdots) \\
& - & (\frac{q^{1/2}+q^{-1/2}}{[1]} x^2e^{-t} + \frac{q+q^{-1}}{[2]} x^2 e^{-2t}
+ \frac{q^{3/2}+q^{-3/2}}{[3]} x^3 e^{-3t} + \cdots) \\
& + & (\frac{q^{1/2}+q^{-1/2}}{[1]} x^2e^{-2t} + \frac{q+q^{-1}}{[2]} x^4 e^{-4t}
+ \frac{q^{3/2}+q^{-3/2}}{[3]} x^6 e^{-6t} + \cdots) \\
& - & (\frac{q+1+q^{-1}}{[1]} x^3e^{-t}
+ \frac{q^2+1+q^{-2}}{[2]} x^6e^{-2t} + \frac{q^3+1+q^{-3}}{[3]} x^9 e^{-3t} + \cdots) \\
& + & (\frac{q^2+q+2+q^{-1}+q^{-2}}{[1]} x^3 e^{-2t} + \cdots) \\
& - & (\frac{q^2+1+q^{-2}}{[1]} x^3 e^{-3t} + \cdots) + \cdots.
\een
In other words,
\begin{align*}
N^{(0)}_{1,0} & = 1, & N^{(0)}_{1,1} & = -1, &
N^{(0)}_{2,0} & = 0, \\
N^{(0)}_{2,1} & = - (q^{1/2}+q^{-1/2}), & N^{(0)}_{2,2} & = q^{1/2}+q^{-1/2}, &
N^{(0)}_{3,0} & = 0, \\
N^{(0)}_{3,1} & = -(q+1+q^{-1}), &
N^{(0)}_{3,2} & = q^2+q+2+q^{-1}+q^{-2}, & N^{(0)}_{3,3} & = - (q^2+1+q^{-2}).
\end{align*}
When $a=1$,
$i$ times the right-hand side of (\ref{eqn:OVg}) becomes:
\ben
&& (\frac{1}{[1]} - \frac{e^{-t}}{[1]}) (-x)
+ (\frac{[3]}{[2][1]} - \frac{[4]e^{-t}}{[1]^2} + \frac{[5]e^{-2t}}{[2][1]}) x^2 \\
& + & (\frac{[4][5]}{[3]!}
- \frac{[5][6]}{[2][1]^2} e^{-t} + \frac{[6][7]}{[2][1]^2}e^{-2t} - \frac{[7][8]}{[3][2][1]} e^{-3t}) (-x)^3
+ \cdots \\
& = & - (\frac{1}{[1]} x + \frac{1}{[2]} x^2 + \frac{1}{[3]} (-x)^3 + \cdots) \\
& + & (\frac{1}{[1]} xe^{-t} + \frac{1}{[2]} x^2e^{-2t} + \frac{1}{[3]} x^3 e^{-3t} + \cdots) \\
& + & (\frac{q^{1/2}+q^{-1/2}}{[1]} x^2 + \frac{q+q^{-1}}{[2]} x^4 + \cdots ) \\
& - & (\frac{q^{3/2}+q^{1/2}+q^{-1/2}+q^{-3/2}}{[1]} x^2e^{-t}
+ \frac{q^{3}+q+q^{-1}+q^{-3}}{[2]} x^4e^{-2t} + \cdots) \\
& + & (\frac{q^{3/2}+q^{-3/2}}{[1]} x^2e^{-2t} + \frac{q^{3}+q^{-3}}{[2]} x^4e^{-4t} + \cdots) \\
& - & (\frac{q^2-q+1-q^{-1}+q^{-2}}{[1]} x^3 + \cdots) + \cdots.
\een
In other words,
\ben
&& N^{(1)}_{1,0} = 1, \qquad N^{(1)}_{1,1} = -1, \qquad
N^{(1)}_{2,0} = q^{1/2}+q^{-1/2}, \\
&& N^{(1)}_{2,1} = - (q^{3/2} + q^{1/2}+q^{-1/2}+q^{-3/2}), \qquad N^{(1)}_{2,2} = q^{3/2}+q^{-3/2}, \\
&& N^{(1)}_{3,0} = q^2-q+1-q^{-1}+q^{-2}, \qquad
N^{(1)}_{3,1} = - \sum_{k=-2}^2 q^k \cdot (q^2+1+q^{-2}), \\
&& N^{(1)}_{3,2} = \sum_{k=-3}^3 q^k \cdot (q^2+1+q^{-2}), \\
&& N^{(1)}_{3,3} = - (q+1+q^{-1})(q-1+q^{-1})(q^3+1+q^{-3}).
\een

\section{The Mirror Curve of the Resolved Conifold}

We apply results on open string one-point functions above to relate them to the mirror curve
of the resolved conifold.

\subsection{The proof of the Conjecture of Aganagic-Vafa}

For $a=0$, we have proved the following identity:
\be
\Psi^{(0)}_{0,1}(t;x) = - \sum_{n=1}^\infty \frac{x^n}{n} \sum_{j=0}^n \frac{(n+j-1)!}{j!j!(n-j)!}Q^j.
\ee
This is exactly the prediction by Aganagic-Vafa \cite{AV}.
From this one can recover the equation of the mirror curve of the resolved conifold with
the outer brane in zero framing as follows.
Aganagic-Vafa \cite{AV} conjectured that the mirror curve when the outer brane is in framing $a$
is described by a plane curve
$$h(x, y) = 0,$$
which determines near $x = 0$ a function $y=y(x; t)$
so that:
\be
x\frac{d}{d x} \Psi_{0,1}^{(a)}(t;x) = \log y(x;t).
\ee
Now we have
\ben
x\frac{\pd}{\pd x} \Psi^{(0)}_{0,1}(t;x)
& = & - \sum_{n=1}^\infty x^n \sum_{j=0}^n \frac{(n+j-1)!}{j!j!(n-j)!}Q^j \\
& = & - \sum_{j, k \geq 0, j+k > 0}^\infty x^{j+k} \frac{(k+2j-1)!}{j!j!k!}Q^j.
\een
One can take the summation in $k$ first to get:
\ben
x\frac{\pd}{\pd x} \Psi^{(0)}_{0,1}(t;x)
& = & \ln (1 - x) - \sum_{j=1}^\infty \frac{(2j-1)!}{j!j!} \frac{(Qx)^j}{(1-x)^{2j}}.
\een
By the well-known series:
\be
\ln (\half + \half \sqrt{1+t^2})
= - \sum_{j=1}^\infty (-1)^j \frac{(2j-1)!}{j!j!} (\frac{t}{2})^{2j},
\ee
we then get:
\be
x\frac{\pd}{\pd x} \Psi^{(0)}_{0,1}(t;x)
= \ln (\half (1 - x) + \half \sqrt{(1-x)^2+4x e^{-t}}).
\ee
It follows that
\be
y:=y(x;t) = \half (1 - x) + \half \sqrt{(1-x)^2+4x e^{-t}},
\ee
and so
\be
y^2 - (1-x) y - xe^{-t} = 0,
\ee
or equivalently,
\be \label{eqn:Mirror0Framing}
x+y-1 - xy^{-1}e^{-t} = 0.
\ee
This is exactly the equation for the mirror curve.
So we have proved

\begin{theorem}
Aganagic-Vafa's Conjectures holds for the case of the resolved conifold \cite[\S 5.1]{AV}
when the outer brane is in zero framing.
\end{theorem}

\subsection{Proof of the Aganagic-Klemm-Vafa conjecture for framing transformation}

The mirror curve for the nonzero framing case is more difficult to find.
It might be possible to generalize the summation method in last subsection by
the theory of hypergeometric series for $a \neq 0$,
but we do not know how to do so.
We will take a more elementary approach.

In \S \ref{sec:g=0} we have seen that the sum in (\ref{eqn:Sum}) is equal to the coefficient of
$z^n$ in $\big(\frac{1+z}{1-Qz} \big)^{(a+1)n}$,
i.e.,
it is equal to
\be
\res_{z=0} \frac{1}{z^{n+1}} \big(\frac{1+z}{1-Qz} \big)^{(a+1)n}
= \frac{1}{2\pi i} \int_{|z| = \epsilon} \frac{1}{z^{n+1}} \big(\frac{1+z}{1-Qz}\big)^{(a+1)n} dz.
\ee
for sufficiently small $\epsilon > 0$ (say $\epsilon < \frac{1}{|Q|}$),
Therefore,
(\ref{eqn:Psi(g=0)}) can be rewritten as follows (we assume $|x|$ is sufficiently small to ensure
uniform convergence):
\ben
(x\frac{\pd}{\pd x})^2 \Psi_{0,1}^{(a)}(t; x)
& = & -  \sum_{n=1}^\infty \frac{x^n}{a+1}
\frac{1}{2\pi i} \int_{|z| = \epsilon} \frac{1}{z^{n+1}} \big(\frac{1+z}{1-Qz}\big)^{(a+1)n} dz \\
& = & - \frac{1}{2\pi (a+1)i} \int_{|z| = \epsilon}  \frac{1}{z} \sum_{n=1}^\infty \frac{x^n}{z^n}
\big(\frac{1+z}{1-Qz}\big)^{(a+1)n} dz \\
& = & - \frac{1}{2\pi (a+1)i} \int_{|z| = \epsilon} \biggl( \frac{1}{z - x \big(\frac{1+z}{1-Qz}\big)^{a+1}}
- \frac{1}{z}  \biggr) dz \\
& = & - \frac{1}{2\pi (a+1)i} \int_{|z| = \epsilon} \biggl( \frac{1}{z - x \big(\frac{1+z}{1-Qz}\big)^{a+1}}
\biggr) dz  + \frac{1}{a+1}.
\een
We use Cauchy's residue theorem to evaluate the contour integral of the analytic function
$$f(z) = \frac{1}{z- x \big(\frac{1+z}{1-Qz}\big)^{a+1}}
= \frac{(1-Qz)^{a+1}}{z(1-Qz)^{a+1} - x (1+z)^{a+1}}.$$
Because $Q$ and $x$ are very small,
so the denominator is close to $z$,
therefore by Rouch\'e's theorem,
there is only one root $z_0=z_0(x_0;Q)$ for $z(1-Qz)^{a+1} - x (1+z)^{a+1}$ inside the disc $|z|< \epsilon$.
This determines $z_0$ as an analytic function in $x$ by the Lagrange inversion formula,
i.e.,
by finding the inverse series of
\be
x = \frac{z_0(1-Qz_0)^{a+1}}{(1+z_0)^{a+1}}.
\ee
The first couple of terms are given by:
\be
z = x + (a+1)(1+Q)x^2 +  \cdots.
\ee
So we get:
\begin{multline} \label{eqn:D^2Psi}
(x\frac{\pd}{\pd x})^2 \Psi_{0,1}^{(a)}(t; x) \\
=  - \frac{1}{a+1} \frac{(1-Qz_0)^{a+1}}{(1-Qz_0)^{a+1}-(a+1)Qz_0(1-Qz_0)^a - (a+1)x(1+z_0)^a} + \frac{1}{a+1}.
\end{multline}
One can use the equation
\be \label{eqn:z0}
z_0(1-Qz_0)^{a+1} - x (1+z_0)^{a+1} = 0.
\ee
to rewrite the right-hand side of (\ref{eqn:D^2Psi}).
Differentiate this equation on both sides with respect to $x$,
we get:
\be \label{eqn:Dz0}
\frac{\pd z_0}{\pd x} = \frac{(1+z_0)^{a+1}}{(1-Qz_0)^{a+1}-(a+1)Qz_0(1-Qz_0)^a - (a+1)x(1+z_0)^a}.
\ee
By (\ref{eqn:z0}) and (\ref{eqn:Dz0}),
(\ref{eqn:D^2Psi}) becomes:
\be
(x\frac{\pd}{\pd x})^2 \Psi_{0,1}^{(a)}(t; x) = - \frac{x}{(a+1)z_0} \frac{\pd z_0}{\pd x} + \frac{1}{a+1}.
\ee
Integrating once:
\be
x\frac{\pd}{\pd x} \Psi_{0,1}^{(a)}(t; x) =\frac{1}{a+1} \log (\frac{x}{z_0})
= \log \frac{1+z_0}{1- Qz_0}.
\ee
According to Aganagic-Vafa \cite{AV},
the mirror curve is determined by the following equation:
\be \label{eqn:YinX}
y:=y(x;t;a) = \frac{1+z_0}{1-Qz_0} = (\frac{x}{z_0})^{1/(a+1)}.
\ee

\begin{theorem}
The mirror curve for the resolved conifold with an outer brane and framing $a$ is given by the following equation:
\be \label{eqn:MirrorAFraming}
y+xy^{-a} - 1 - e^{-t}xy^{-a-1} = 0.
\ee
\end{theorem}

\begin{proof}
By (\ref{eqn:YinX}) we get:
\be \label{eqn:z0inXY}
z_0 = x y^{-a-1}.
\ee
One can rewrite (\ref{eqn:z0}) as
\be
(\frac{x}{z_0})^{1/(a+1)} = \frac{1-Qz_0}{1+z_0}.
\ee
Plug in (\ref{eqn:YinX}) and (\ref{eqn:z0inXY}):
\be
y = \frac{1 - Q xy^{-a-1}}{1 + xy^{-a-1}}.
\ee
This is equivalent to (\ref{eqn:MirrorAFraming}).
\end{proof}

In \cite{AKV},
it was conjectured that the equation of the mirror curve in zero framing and the $a$ framing are related by
the following framing transformation:
\begin{align}
x & \mapsto xy^{-a}, & y & \mapsto y.
\end{align}
Note (\ref{eqn:Mirror0Framing}) is transformed to (\ref{eqn:MirrorAFraming}).
So we have established the Aganagic-Klemm-Vafa conjecture for framing transformation
in the case of the resolved conifold with one outer brane.

\subsection{The quantum mirror curve}
Because
\be
\lambda \Psi^{(a)}_1(\lambda; t; x_1)
= \sum_{g \geq 0} \lambda^{2g} \Psi^{(a)}_{g,1}(t; x)
\ee
is a deformation of $\Psi^{(a)}_{0,1}(t;x)$, with the parameter $\lambda$ serving as the Planck constant,
if we set
\be
y(\lambda;t,x) = \exp (\lambda \cdot x\frac{d}{d x} \Psi_{1}^{(a)}(\lambda; t;x)),
\ee
then $y(\lambda;t;x)$ is a deformation of $y(t;x)$.
The family of curves $y = y(\lambda; t; x)$ can be understood as the ``quantum mirror curve"
of the resolved conifold.

In \S \ref{sec:g>0} we have seen that the sum in (\ref{eqn:Sum(g)}) is equal to the coefficient of
$z^n$ in
\ben
&& f^{(a)}_n(z;t)
=
\begin{cases}
\prod_{k=1}^{(a+1)n} \frac{1 + q^{((a+1)n-2k+1)/2}z}{1-Qq^{((a+1)n-2k+1)/2}z}, & a \geq 0, \\
\prod_{k=1}^{bn} \frac{1-Qq^{(bn-2k+1)/2}z}{1 + q^{(bn-2k+1)/2}z}, & a+1=-b \leq -1.
\end{cases}
\een
We will treat the case of $a \geq 0$ here,
the $a \leq -2$ case is similar.
The coefficient is equal to
\ben
&& \res_{z=0} \prod_{k=1}^{(a+1)n} \frac{1 + q^{((a+1)n-2k+1)/2}z}{1-Qq^{((a+1)n-2k+1)/2}z} \\
& = & \frac{1}{2\pi i} \int_{|z| = \epsilon} \frac{1}{z^{n+1}} \prod_{k=1}^{(a+1)n}
\frac{1 + q^{((a+1)n-2k+1)/2}z}{1-Qq^{((a+1)n-2k+1)/2}z} dz,
\een
for sufficiently small $\epsilon > 0$, now depending on $n$.
To avoid difficulties in analysis,
we will understand the residue formally as taking the coefficient of $z^{-1}$.
Now (\ref{eqn:Psi(g,1)}) can be rewritten as follows:
\be
x\frac{\pd}{\pd x} \Psi_{1}^{(a)}(\lambda;t; x)
= \res_{z=0}  \sum_{n=1}^\infty \frac{x^n}{i \cdot [(a+1) n]}
\prod_{k=1}^{(a+1)n} \frac{1 + q^{((a+1)n-2k+1)/2}z}{1-Qq^{((a+1)n-2k+1)/2}z}.
\ee
We do not know how to take the summation on the right-hand side at present.

\end{document}